\documentclass[a4paper]{amsart}
\usepackage{amssymb}
\usepackage{amscd}
\usepackage{amsthm}
\usepackage{amsmath}
\usepackage{latexsym}
\usepackage[all]{xy}

\theoremstyle{plain}
\newtheorem{theorem}{Theorem}
\newtheorem{corollary}[theorem]{Corollary}
\newtheorem{lemma}[theorem]{Lemma}

\theoremstyle{definition}
\newtheorem{definition}[theorem]{Definition}
\newtheorem{example}[theorem]{Example}

\theoremstyle{remark}

\newtheorem{remark}[theorem]{Remark}

\numberwithin{theorem}{section}

\newcommand{\card}[1]{\mbox{\rm{card}}(#1)}

\newcommand{\mSpec}[1]{\mbox{\rm{mSpec}}(#1)}

\newcommand{\Soc}[1]{\mbox{\rm{Soc}}(#1)}

\newcommand{\im}[1]{\mbox{\rm{Im}}(#1)}
\newcommand{\cf}[1]{\mbox{\rm{cf}}(#1)}

\newcommand{\Hom}[3]{\mbox{\rm{Hom}}_{#1}(#2,#3)}
\newcommand{\Ext}[4]{\mbox{\rm{Ext}}^{#1}_{#2}(#3,#4)}

\newcommand{\rmod}[1]{\mbox{\rm{Mod}--}{#1}}

\begin{document}

\title{Faith's problem on $R$-projectivity is undecidable}
\author{\textsc{Jan Trlifaj}}
\address{Charles University, Faculty of Mathematics and Physics, Department of Algebra \\
Sokolovsk\'{a} 83, 186 75 Prague 8, Czech Republic}
\email{trlifaj@karlin.mff.cuni.cz}

\date{\today}
\subjclass[2010]{Primary: 16D40, 03E35. Secondary: 16E30, 16E50, 03E45, 18G05.}
\keywords{Baer Criterion, Jensen-functions, non-perfect rings, projective module, $R$-projective module.}
\dedicatory{In memory of Gena Puninski.}
\thanks{Research supported by GA\v CR 17-23112S}
\begin{abstract} In \cite{F}, Faith asked for what rings $R$ does the Dual Baer Criterion hold in $\rmod R$, that is, when does $R$-projectivity imply projectivity for all right $R$-modules? Such rings $R$ were called right testing. Sandomierski proved that if $R$ is right perfect, then $R$ is right testing. Puninski et al.\ \cite{AIPY} have recently shown for a number of non-right perfect rings that they are not right testing, and noticed that \cite{T2} proved consistency with ZFC of the statement {\lq}each right testing ring is right perfect{\rq} (the proof used Shelah's uniformization). 

Here, we prove the complementing consistency result: the existence of a right testing, but not right perfect ring is also consistent with ZFC (our proof uses Jensen-functions). Thus the answer to the Faith's question above is undecidable in ZFC. We also provide examples of non-right perfect rings such that the Dual Baer Criterion holds for {\lq}small{\rq} modules (where {\lq}small{\rq} means countably generated, or $\leq 2^{\aleph_0}$-presented of projective dimension $\leq 1$). 

\end{abstract}

\maketitle

\section{Introduction}

The classic Baer Criterion for Injectivity \cite{B} says that a (right $R$-) module $M$ is injective, if and only if it is \emph{$R$-injective}, that is, each homomorphism from any right ideal $I$ of $R$ into $M$ extends to $R$. This criterion is the key tool for classification of injective modules over particular rings. 

A module $M$ is called \emph{$R$-projective} provided that each homomorphism from $M$ into $R/I$ where $I$ is any right ideal, factors through the canonical projection $\pi : R \to R/I$ \cite[p.184]{AF}. One can formulate the \emph{Dual Baer Criterion} as follows: a module $M$ is projective, if and only if it is $R$-projective. The rings $R$ such that this criterion holds true are called right \emph{testing}, \cite[Definition 2.2]{AIPY}.  

Dualizations are often possible over perfect rings. Indeed, Sandomierski proved that each right perfect ring is right testing \cite{S}. The question of existence of non-right perfect right testing rings is much harder. Faith \cite[p.175]{F} says that {\lq\lq}the characterization of all such rings is still an open problem{\rq\rq} -- we call it the Faith's problem here. 

Note that if $R$ is not right perfect, then it is consistent with ZFC + GCH that $R$ is not right testing. Indeed, as observed in \cite{AIPY}, \cite[Lemma 2.4]{T2} (or \cite{T1}) implies that there is a $\kappa^+$-presented module $N$ of projective dimension $1$ such that $\Ext 1RN{I} = 0$ for each right ideal $I$ of $R$ (and hence $N$ is $R$-projective, but not projective) in the extension of ZFC satisfying GCH and Shelah's Uniformization Principle UP$_\kappa$ for an uncountable cardinal $\kappa$ such that $\card R < \kappa$ and $\cf{\kappa} = \aleph_0$. In particular, attempts \cite{DHH} to prove the existence of non-right perfect testing rings in ZFC could not be successful. 

Moreover, in the extension of ZFC + GCH satisfying UP$_\kappa$ for all uncountable cardinals $\kappa$ such that $\cf{\kappa} = \aleph_0$ \cite{ES}, all right testing rings are right perfect. So it is consistent with ZFC + GCH that all right testing rings are right perfect.  

For many non-right perfect rings $R$, one can actually prove that $R$ is not right testing in ZFC: this is the case for all commutative noetherian rings \cite[Theorem 1]{H}, all semilocal right noetherian rings \cite[Proposition 2.11]{AIPY}, and all commutative domains (see Lemma \ref{domains} below). 

It is easy to see that all finitely generated $R$-projective modules are projective, that is, the Dual Baer Criterion holds for all finitely generated modules over any ring. So in order to find examples of $R$-projective modules which are not projective, one has to deal with infinitely generated modules. The task is quite complex in general: in Section \ref{inquiry}, we will show that there exist non-right perfect rings such that the Dual Baer Criterion holds for all countably generated modules, or for all $\leq 2^{\aleph_0}$-presented modules of projective dimension $\leq 1$. 

Some questions related to the vanishing of Ext, such as the Whitehead problem, are known to be undecidable in ZFC, cf.\ \cite{EM}. In Section \ref{independence}, we will prove that this is also true of the existence of non-right perfect right testing rings. To this purpose, we will employ G\" odel's Axiom of Constructibility V = L, or rather its combinatorial consequence, the existence of Jensen-functions (see \cite[\S VI.1]{EM} and \cite[\S 18.2]{GT}). Our main result, Theorem \ref{independence} below, says that the existence of Jensen-functions implies that a particular subring of $K^\omega$ (where $K$ is a field of cardinality $\leq 2^\omega$) is testing, but not perfect. 

For unexplained terminology, we refer the reader to \cite{AF}, \cite{EM}, \cite{GT} and \cite{G}.  

\section{$R$-projectivity versus projectivity}\label{inquiry}

It is easy to see that for each $R$-projective module $M$, each submodule $I \subseteq R^n$ and each $f \in \Hom RM{R^n/N}$, there exists $g \in \Hom RM{R^n}$ such that $f = \pi_N g$ where $\pi _N : R^n \to R^n/N$ is the projection (see e.g.\ \cite[Proposition 16.12(2)]{AF}). In particular, all finitely generated $R$-projective modules are projective. 

This not true of countable generated $R$-projective modules in general - for example, by the following lemma, the abelian group $\mathbb Q$ is $R$-projective, but not projective:

\begin{lemma}\label{domains} Let $R$ be a commutative domain. Then each divisible module is $R$-projective. So $R$ is testing, iff $R$ is a field.
\end{lemma}
\begin{proof} Assume $R$ is testing and possesses a non-trivial ideal $I$. Let $M$ be any divisible module.  If $0 \neq \Hom RM{R/I}$, then $R/I$ contains a non-zero divisible submodule of the form $J/I$ for an ideal $I \subsetneq J \subseteq R$. Let $0 \neq r \in I$. The $r$-divisibility of $J/I$ yields $Jr + I = J$, but $Jr \subseteq I$, a contradiction. So $\Hom RM{R/I} = 0$, and $M$ is projective. In particular, each injective module is projective, so $R$ is a commutative QF-domain, hence a field.   
\end{proof}

However, there do exist rings such that all countably generated $R$-projective modules are projective. We will now examine one such class of rings that will be relevant for proving the independence result in Section \ref{Faith}:

\begin{definition}\label{bergman} Let $K$ be a field, and $R$ the unital $K$-subalgebra of $K^\omega$ generated by $K^{(\omega)}$. In other words, $R$ is the subalgebra of $K^\omega$ consisting of all eventually constant sequences in $K^\omega$. 

For each $i < \omega$, we let $e_i$ be the idempotent in $K^\omega$ whose $i$th component is $1$ and all the other components are $0$. Notice that $\{ e_i \mid i < \omega \}$ is a set of pairwise orthogonal idempotents in $R$, so $R$ is not perfect. 
\end{definition}

First, we note basic ring and module theoretic properties of this particular setting:

\begin{lemma}\label{basic}  Let $R$ be as in Definition \ref{bergman}.
\begin{enumerate}
\item $R$ is a commutative von Neumann regular semiartinian ring of Loewy length $2$, with $\Soc R = \sum_{i < \omega} e_i R = K^{(\omega)}$ and $R/\Soc R \cong K$.
\item  If $I$ is an ideal of $R$, then either $I = I_A = \sum_{i \in A} e_i R$ for a subset $A \subseteq \omega$ and $I$ is semisimple and projective, or else $I = fR$ for an idempotent $f \in R$ such that $f$ is eventually $1$. In particular, $R$ is hereditary.
\item $\{ e_iR \mid i < \omega \} \cup \{ S \}$ is a representative set of all simple modules, where $S = R/\Soc{R}$. All these modules are $\sum$-injective, and all but $S$ are projective. 
\item Let $M \in \rmod R$. Then there are unique cardinals $\kappa$, $\kappa_i$ ($i < \omega$) and $\lambda$ such that $M \cong S^{(\kappa)} \oplus N$, $\Soc N \cong \bigoplus_{i < \omega} (e_iR)^{(\kappa_i)}$, and $N/\Soc{N} \cong S^{(\lambda)}$.  

If $N = R^{(\mu)}/I$, then 
$$\Soc N = (\Soc{R^{(\mu)}} + I)/I \cong \Soc{R^{(\mu)}}/(\Soc{R^{(\mu)}} \cap I)$$ and $N/\Soc N \cong R^{(\mu)}/(\Soc{R^{(\mu)}} + I)$. Hence for each $i < \omega$, $\kappa_i$ is the codimension of the $e_iR$-homogenous component of $\Soc{R^{(\mu)}} \cap I$ in $\Soc{R^{(\mu)}}$, while $\lambda$ is the codimension of  $(\Soc{R^{(\mu)}} + I)/\Soc{R^{(\mu)}}$ in $R^{(\mu)}/\Soc{R^{(\mu)}} \cong S^{(\mu)}$.
\end{enumerate}
\end{lemma}
\begin{proof} (1) Clearly, $R$ is commutative, and if $r \in R$, then all non-zero components of $r$ are invertible in $K$, so there exists $s \in R$ with $rsr = r$, i.e., $R$ is von Neumann regular. 

For each $i < \omega$, $e_iR = e_i K^\omega$ is a simple projective module, whence $J = \sum_{i < \omega} e_i R \subseteq \Soc R$. Moreover, $R/J \cong K$ is a simple non-projective module. So $R$ is semiartinian of Loewy length $2$, and $J = \Soc R$ is a maximal ideal of $R$.    

(2) If $I \subseteq \Soc R$, then $I$ is a direct summand in the semisimple projective module $\Soc R$. Since the simple projective modules $\{ e_iR \mid i < \omega \}$ are pairwise non-isomorphic, $I \cong I_A = \sum_{i \in A} e_i R$, and hence $I = I_A$, for a subset $A \subseteq \omega$.

If $I \nsubseteq \Soc R$, then there is an idempotent $e \in I \setminus \Soc R$ and $eR + \Soc R = R$. Note that $e$ is eventually $1$, so 
in particular, $eR \supseteq \sum_{i \in B} e_i R$ where $B \subseteq \omega$ is the (cofinite) set of all indices $i$ such that the $i$th component of $e$ is $1$. Then $I = eR \oplus (\sum_{i \notin B} e_i R \cap I)$. The latter direct summand equals $I_A$ for a (finite) subset $A \subseteq \omega \setminus B$, and $I = fR$ for the idempotent $f = e + \sum_{i \in A} e_i$.         

In either case, $I$ is projective, hence $R$ is hereditary.

(3) By part (2), the maximal spectrum $\mSpec R = \{ I_\omega \} \cup \{ (1 - e_i)R \mid i < \omega \}$. The $\sum$-injectivity of all simple modules follows from part (1) and \cite[Proposition 6.18]{G}. The simple module $S$ is not projective because $I_\omega$ is not finitely generated.  

(4) These (unique) cardinals are determined as follows: $\kappa$ is the dimension of the $S$-homogenous component of $M$, and $\kappa_i$ the dimension of its $e_iR$-homogenous component ($i < \omega$). The semisimple module $\bar M = M/\Soc M \cong N/\Soc N$ is isomorphic to a direct sum of copies of the unique non-projective simple module $S$; $\lambda$ is the ($S$-) dimension of $\bar M$.  

The final claim follows from the fact that $P = (\Soc{R^{(\mu)}} + I)/I$ is a direct sum of projective simple modules, while  $R^{(\mu)}/(\Soc{R^{(\mu)}} + I)$ a direct sum of copies of $S$, so $\{ 0 , P, N \}$ is the socle sequence of $N$.
\end{proof}

Next we turn to $R$-projectivity:

\begin{lemma}\label{R-proj}   Let $R$ be as in Definition \ref{bergman}.
\begin{enumerate}
\item A module $M$ is $R$-projective, iff it is projective w.r.t.\ the projection $\pi : R \to R/\Soc{R}$. 
\item The class of all $R$-projective modules is closed under submodules. If $M \in \rmod R$ is $R$-projective, then all countably generated submodules of $M$ are projective. In particular, the Dual Baer Criterion holds for all countably generated modules.
\end{enumerate}
\end{lemma}
\begin{proof} (1) First, note that by part (2) of Lemma \ref{basic}, the only ideals $I$ such that $R/I$ is not projective, are of the form $I = I_A$ where $A$ is an infinite subset of $\omega$ (and hence $I \subseteq \Soc{R} = I_{\omega}$). So it suffices to prove that if $M$ is projective w.r.t.\ the projection $\pi : R \to R/\Soc{R}$, then it is projective w.r.t.\ all the projections $\pi_{I_A} : R \to R/I_A$ such that $A \subseteq \omega$ is infinite.

Let $f \in \Hom RM{R/I_A}$. If $\im{f }\subseteq \Soc{R}/I_A$, then there exists a homomorphism $h \in  \Hom R{\Soc{R}/I_A}{\Soc R}$ such that $\pi_{I_A} h = id$, whence $g = hf$ yields a factorization of $f$ through $\pi_{I_A}$. Otherwise, let $\rho : R/I_A \to R/\Soc{R}$ be the projection. By assumption, there is $g \in \Hom RMR$ such that $\rho f = \pi g$. So $\rho (f - \pi_{I_A} g) = 0$, and $\im{f - \pi_{I_A} g} \subseteq \Soc{R}/I_A$. Then $f - \pi_{I_A} g$ factorizes through $\pi_{I_A}$ by the above, and so does $f$.    

(2) The closure of the class of all $R$-projective modules under submodules follows from part (1) and from the injectivity of $S = R/\Soc R$ (see part (3) of Lemma \ref{basic}). So it only remains to prove that each countably generated $R$-projective module is projective. However, as remarked above, for any ring $R$, each finitely generated $R$-projective module is projective. Since $R$ is hereditary and von Neumann regular, \cite[Lemma 3.4]{T2} applies and gives that also all countably generated $R$-projective modules are projective. 
\end{proof}

\medskip
We finish this section by presenting two more classes of non-right perfect rings over which small modules satisfy the Dual Baer Criterion.

In both cases, the rings will be von Neumann regular and right self-injective. Apart from classic facts about these rings from \cite[\S10]{G}, we will also need the following easy observation (valid for any right self-injective ring $R$, see \cite[Proposition 2.6]{AIPY}): a module $M$ is $R$-projective, iff $\Ext 1RMI = 0$ for each right ideal $I$ of $R$.       

\begin{example}\label{commutat} 
Let $R$ be a right self-injective von Neumann regular ring such that $R$ has primitive factors artinian, but $R$ is not artinian (e.g., let $R$ be an infinite direct product of skew-fields). Then all $R$-projective modules are non-singular, and the Dual Baer Criterion holds for all countably generated modules. 

For the first claim, let $M$ be $R$-projective and assume there is an essential right ideal $I \subsetneq R$ such that $R/I$ embeds into $M$.
Let $J$ be a maximal right ideal containing $I$. By \cite[Proposition 6.18]{G}, the simple module $R/J$ is injective, so the projection $\rho : R/I \to R/J$ extends to some $f \in \Hom RM{R/J}$. The $R$-projectivity of $M$ yields $g \in \Hom RMR$ such that $f = \pi g$ where $\pi: R \to R/J$ is the projection. Then $g$ restricts to a non-zero homomorphism from $R/I$ into the non-singular module $R$, a contradiction. Thus, $M$ is non-singular.

For the second claim, we recall from \cite[Example 6.8]{HT}, that for von Neumann regular right self-injective rings, non-singular modules coincide with the (flat) Mittag-Leffler ones. However, each countably generated flat Mittag-Leffler module (over any ring) is projective, see e.g.\ \cite[Corollary 3.19]{GT}. Thus each countably generated $R$-projective module is projective.  
\end{example}     

\begin{example}\label{pureinf} Let $R$ be a von Neumann regular right self-injective ring which is \emph{purely infinite} in the sense of \cite[Definition on p.116]{G}. That is, there exists no central idempotent $0 \neq e \in R$ such that the ring $eRe$ is directly finite (where a ring $R$ is \emph{directly finite} in case $xy = 1$ implies $yx = 1$ for all $x, y \in R$.) 

For example, the endomorphism ring of any infinite dimensional right vector space over a skew-field has this property, see \cite[p. 116]{G}.   

We claim that the Dual Baer Criterion holds for all $\leq 2^{\aleph_0}$-presented modules $M$ of projective dimension $\leq 1$. Indeed, assume that such module $M$ is $R$-projective. By \cite[Theorem 10.19]{G}, $R$ contains a right ideal $J$ which is a free module of rank $2^{\aleph_0}$. If the projective dimension of $M$ equals $1$, then there is a non-split presentation $0 \to K \to L \to M \to 0$ where $K$ and $L$ are free of rank $\leq 2^{\aleph_0}$. Thus $\Ext 1RMJ \neq 0$, in contradiction with the $R$-projectivity of $M$. This shows that $M$ is projective.  

In particular, if the global dimension of $R$ is $2$, and all right ideals of $R$ are $\leq 2^{\aleph_0}$-presented (which is the case when $R$ is the endomorphism ring of a vector space of dimension $\aleph_0$ over a field of cardinality $\leq 2^{\aleph_0}$ under CH - see \cite{O}), then the Dual Baer Criterion holds for all ideals of $R$.  
\end{example}

\begin{remark} As mentioned in the Introduction, for any non-right perfect ring $R$, Shelah's Uniformization Principle UP$_\kappa$ (for an uncountable cardinal $\kappa$ such that $\card R < \kappa$ and $\cf{\kappa} = \aleph_0$) and GCH imply the existence of a $\kappa^+$-presented $R$-projective module $N$ of projective dimension equal to $1$. 

If we choose $R$ to be the endomorphism ring of a vector space of dimension $< \aleph_{\omega}$ over a field of cardinality $< \aleph_{\omega}$, then we can take the minimal choice, $\kappa = \aleph_{\omega}$, so the module $N$ above can be chosen $\aleph_{\omega}^+$-presented. Example \ref{pureinf} gives a lower bound for the possible size of $N$: it has to be $> 2^{\aleph_0}$-presented.               
\end{remark}

\section{The consistency of existence of non-perfect testing rings}\label{Faith}

In this section, we return to the setting of Definition \ref{bergman}, so $K$ will denote a field, and $R$ the subalgebra of $K^\omega$ consisting of all eventually constant sequences in $K^\omega$.  In order to prove that it is consistent with ZFC that $R$ is testing, we will employ the notion of Jensen-functions, cf.\ \cite{J} and \cite[\S 18.2]{GT}:

\begin{definition}\label{reg} Let $\kappa$ be a regular uncountable cardinal. 
\begin{enumerate} 
\item A subset $C \subseteq \kappa$ is called a \emph{club} provided that $C$ is \emph{closed} in $\kappa$ (i.e., $\hbox{sup}(D) \in C$ for each subset $D \subseteq C$ such that $\hbox{sup}(D) < \kappa$) and $C$ is \emph{unbounded} (i.e., $\hbox{sup}(C) = \kappa$). Equivalently, there exists a strictly increasing continuous function $f : \kappa \to \kappa$ whose image is $C$. 
\item A subset $E \subseteq \kappa$ is \emph{stationary} provided that $E \cap C \neq \emptyset$ for each club $C \subseteq \kappa$.
\item Let $A$ be a set of cardinality $\leq \kappa$. An increasing continuous chain, $\{ A_\alpha \mid \alpha < \kappa \}$, consisting of subsets of $A$ of cardinality $< \kappa$ such that $A_0 = 0$ and $A = \bigcup_{\alpha < \kappa} A_\alpha$, is called a \emph{$\kappa$-filtration} of the set $A$. 
\item Let $E$ be a stationary subset of $\kappa$. Let $A$ and $B$ be sets of cardinality $\leq \kappa$. Let $\{ A_\alpha \mid \alpha < \kappa \}$ and $\{ B_\alpha \mid \alpha < \kappa \}$ ) be $\kappa$-filtrations of $A$ and $B$, respectively. For each $\alpha < \kappa$, let $c_\alpha: A_\alpha \to B_\alpha$ be a map. Then $( c_\alpha \mid \alpha < \kappa )$ are called \emph{Jensen-functions} provided that for each map $c : A \to B$, the set $E(c) = \{ \alpha \in E \mid c \restriction A_\alpha = c_\alpha \}$ is stationary in $\kappa$.
\end{enumerate}
\end{definition}

Jensen \cite{J} proved the following (cf.\ \cite[Theorem 18.9]{GT})

\begin{theorem}\label{jensen} Assume G\" odel's Axiom of Constructibility (V = L). Let $\kappa$ be a regular infinite cardinal, $E \subseteq \kappa$ a stationary subset of $\kappa$, and $A$ and $B$ sets of cardinality $\leq \kappa$. Let $\{ A_\alpha \mid \alpha < \kappa \}$ and $\{ B_\alpha \mid \alpha < \kappa \}$ ) be $\kappa$-filtrations of $A$ and $B$, respectively. Then there exist Jensen-functions $( c_\alpha \mid \alpha < \kappa )$.
\end{theorem}
                 
Now, we can prove our main result:								

\begin{theorem}\label{independence} Assume V = L. Let $K$ be a field of cardinality $\leq 2^{\omega}$. Then all $R$-projective modules are projective.
\end{theorem}
\begin{proof} Let $M$ be an $R$-projective module. By induction on the minimal number of generators, $\kappa$, of $M$, we will prove that
$M$ is projective. For $\kappa \leq \aleph_0$, we appeal to part (2) of Lemma \ref{R-proj}, and for $\kappa$ a singular cardinal, we apply \cite[Corollary 3.11]{T2}.  

Assume $\kappa$ is a regular uncountable cardinal. Let $G = \{ m_\alpha \mid \alpha < \kappa \}$ be a minimal set of $R$-generators of $M$. For each $\alpha < \kappa$, let $G_\alpha = \{ m_\beta \mid \beta < \alpha \}$. Let $M_\alpha$ be the submodule of $M$ generated by $G_\alpha$. Then $\mathcal M = (M_\alpha \mid \alpha < \kappa )$ is a $\kappa$-filtration of the module $M$. Possibly skipping some terms of $\mathcal M$, we can w.l.o.g.\ assume that $\mathcal M$ has the following property for each $\alpha < \kappa$: 
if $M_{\beta}/M_\alpha$ is not $R$-projective for some $\alpha < \beta < \kappa$, then also $M_{\alpha + 1}/M_\alpha$ is not $R$-projective. 
Let $E$ be the set of all $\alpha < \kappa$ such that $M_{\alpha + 1}/M_\alpha$ is not $R$-projective.

We claim that $E$ is not stationary in $\kappa$. If our claim is true, then there is a club $C$ in $\kappa$ such that $C \cap E = \emptyset$. Let $f : \kappa \to \kappa$ be a strictly increasing continuous function whose image is $C$. Then $M_{f(\alpha + 1)}/M_{f(\alpha )}$ is $R$-projective for each $\alpha < \kappa$. By the inductive premise, $M_{f(\alpha + 1)}/M_f(\alpha )$ is projective for all $\alpha < \kappa$, whence $M$ is projective, too.

Assume our claim is not true. We will make use of Theorem \ref{jensen} in the following setting. We let $A = G$ and $B = R$. The relevant $\kappa$-filtration of $A$ will be $( G_\alpha \mid \alpha < \kappa)$. For $B$, we consider any $\kappa$-filtration $( R_\alpha \mid \alpha < \kappa )$ of the additive group $(R,+)$ consisting of subgroups of $(R,+) $ (which exists since $\card K \leq \aleph_1$ implies $\card R \leq \aleph_1 \leq \kappa$; if $\card{K}$ is countable, the filtration can even be taken constant $= R$). By Theorem \ref{jensen}, there exist Jensen-functions $c_\alpha : G_\alpha \to R_\alpha$ ($\alpha < \kappa$) such that for each function $c: G \to R$, the set $E(c) = \{ \alpha \in E \mid c_\alpha = c \restriction G_\alpha \}$ is stationary in $\kappa$.
 
By induction on $\alpha < \kappa$, we will define a sequence $( g_\alpha \mid \alpha < \kappa )$ such that $g_\alpha \in \Hom R{M_\alpha}S$ as follows: $g_0 = 0$; if $\alpha < \kappa$ and $g_\alpha$ is defined, we distinguish two cases: 

(I) $\alpha \in E$, and there exist $h_{\alpha +1} \in \Hom R{M_{\alpha+1}}S$ and $y_{\alpha +1} \in \Hom R{M_{\alpha+1}}R$, such that $h_{\alpha +1} \restriction M_\alpha = g_\alpha$, $h_{\alpha +1} = \pi y_{\alpha +1}$ and $y_{\alpha +1} \restriction G_{\alpha} = c_\alpha$. In this case we define $g_{\alpha +1} = h_{\alpha +1} + f_{\alpha +1} \rho_{\alpha +1}$, where $\rho_{\alpha +1} : M_{\alpha +1} \to M_{\alpha +1}/M_\alpha$ is the projection and $f_{\alpha +1} \in \Hom R{M_{\alpha+1}/M_\alpha}S$ is chosen so that it does not factorize through $\pi$ (such $f_{\alpha +1}$ exists because $\alpha \in E$ by part (1) of Lemma \ref{R-proj}. Note that $g_{\alpha+1} \restriction M_\alpha = h_{\alpha +1} \restriction M_\alpha = g_\alpha$. 

(II) otherwise. In this case, we let $g_{\alpha+1} \in \Hom R{M_{\alpha+1}}S$ be any extension of $g_\alpha$ to $M_{\alpha +1}$ (which exists by the injectivity of $S$).

If $\alpha < \kappa$ is a limit ordinal, we let $g_\alpha = \bigcup_{\beta < \alpha} g_\beta$. Finally, we define $g = \bigcup_{\alpha < \kappa} g_\alpha$. We will prove that $g$ does not factorize through $\pi$. This will contradict the $R$-projectivity of $M$, and prove our claim.

Assume there is $x \in \Hom RMR$ such that $g = \pi x$. Then the set of all $\alpha < \kappa$ such that $x \restriction G_{\alpha}$ maps into $R_\alpha$ is closed and unbounded in $\kappa$, so it contains some element $\alpha \in E(x \restriction G)$. For such $\alpha$, we have $g_{\alpha +1} = \pi x \restriction M_{\alpha + 1}$ and $x \restriction G_{\alpha} = c_\alpha$, so $\alpha$ is in case (I) (this is witnessed by taking $h_{\alpha + 1} = g_{\alpha + 1}$ and $y_{\alpha + 1} = x \restriction M_{\alpha +1}$).  

Let $z_{\alpha +1} = x \restriction M_{\alpha + 1} - y_{\alpha +1}$. Then $z_{\alpha +1} \restriction G_{\alpha} = x \restriction G_{\alpha} - y_{\alpha +1} \restriction G_\alpha = c_\alpha - c_\alpha = 0$. So there exists $u_{\alpha +1} \in \Hom R{M_{\alpha +1}/M_\alpha}R$ such that $z_{\alpha +1} = u_{\alpha +1} \rho_{\alpha +1}$. Moreover, 
$$\pi u_{\alpha +1} \rho_{\alpha +1} = \pi z_{\alpha +1} = \pi x \restriction M_{\alpha + 1} - \pi y_{\alpha +1} = g_{\alpha + 1} - h_{\alpha +1} = f_{\alpha +1} \rho_{\alpha +1}.$$
Since $\rho_{\alpha +1}$ is surjective, we conclude that $\pi u_{\alpha +1} = f_{\alpha +1}$, in contradiction with our choice of the homomorphism $f_{\alpha +1}$.  
\end{proof}           

\begin{corollary}\label{undec} Let $K$ be a field of cardinality $\leq 2^{\omega}$. Then the statement {\lq}$R$ is a testing ring{\rq} is independent of ZFC + GCH. Hence Faith's problem is undecidable in ZFC + GCH.
\end{corollary}
\begin{proof} Assume UP$_\kappa$ for some $\kappa$ such that $\card R < \kappa$ and $\cf{\kappa} = \aleph_0$. Then $R$ is not testing by 
\cite[Lemma 2.4]{T2} (see also \cite[Theorem 2.7]{AIPY}).

Assume V = L. Then $R$ is testing by Theorems \ref{jensen} and \ref{independence}.
\end{proof}

{\bf Acknowledgement:} I owe my thanks to Gena Puninski for drawing my attention to Faith's problem, and for sharing his manuscript \cite{P} (later incorporated in \cite{AIPY}). I deeply regret Gena's sudden departure in April 2017.

\end{document}